\documentclass[11pt]{amsart}
\usepackage[mathscr]{eucal}
\usepackage{amssymb, amsmath,relsize,array, amscd}
\usepackage{url}
\usepackage[plainpages]{hyperref}

\input xy
\xyoption{all}

\newtheorem{theorem}{Theorem}[section]
\newtheorem*{theo}{Theorem}

\newtheorem{lemma}[theorem]{Lemma}
\newtheorem{proposition}[theorem]{Proposition}

\theoremstyle{definition}
\newtheorem*{defn}{Definition}
\newtheorem{remark}[theorem]{Remark}

\numberwithin{equation}{section}

\renewcommand{\k}{\Bbbk}

\newcommand{\set}[1]{\left\{#1\right\}}
\newcommand{\C}{\mathbb{C}}
\newcommand{\PP}{\mathbb{P}}
\newcommand{\R}{\mathbb{R}}
\newcommand{\Q}{\mathbb{Q}}
\newcommand{\Z}{\mathbb{Z}}

\DeclareMathOperator{\id}{id}
\DeclareMathOperator{\tc}{{\sf TC}}
\DeclareMathOperator{\cl}{\sf cl}
\DeclareMathOperator{\zcl}{\sf zcl}
\DeclareMathOperator{\cat}{cat}
\DeclareMathOperator{\secat}{\sf {genus}}
\DeclareMathOperator{\PSL}{PSL}
\DeclareMathOperator{\SO}{SO}

\begin{document}

\title[Collision-free motion planning on surfaces]
{Topological complexity of collision-free motion planning on surfaces}

\author[D. Cohen]{Daniel C. Cohen}
\address{Department of Mathematics, Louisiana State University,
Baton Rouge, LA 70803, USA}
\email{\href{mailto:cohen@math.lsu.edu}{cohen@math.lsu.edu}}
\urladdr{\href{http://www.math.lsu.edu/~cohen/}
{www.math.lsu.edu/\char'176cohen}}

\author[M. Farber]{Michael Farber}
\address{Department of Mathematics, University of Durham, Durham DH1 3LE, UK}
\email{\href{mailto:Michael.Farber@durham.ac.uk}{Michael.Farber@durham.ac.uk}}
\urladdr{\href{http://maths.dur.ac.uk/~dma0mf/}
{maths.dur.ac.uk/\char'176dma0mf}}



\begin{abstract}
The topological complexity $\tc(X)$ is a numerical homotopy invariant of a topological space $X$ which is motivated by robotics and is similar in spirit to the classical Lusternik--Schnirelmann category of $X$.  Given a mechanical system with configuration space $X$, the invariant $\tc(X)$ 
measures the complexity of all possible motion planning algorithms designed for the system. 

In this paper, we compute the topological complexity of the configuration space of $n$ distinct ordered points on an orientable surface. Our main tool is a theorem of B.~Totaro describing the cohomology of configuration spaces of algebraic varieties. 
\end{abstract}


\maketitle

\section{Introduction} \label{sec:intro}

Let $X$ be a path-connected topological space, with the homotopy type of a finite CW-complex.  Viewing $X$ as the space of configurations of a mechanical system, the motion planning problem consists of constructing an algorithm which takes as input pairs of configurations $(x_0,x_1) \in X \times X$, and produces a continuous path $\gamma\colon [0,1] \to X$ from the initial configuration $x_0=\gamma(0)$ to the terminal configuration $x_1=\gamma(1)$.  The motion planning problem is a central theme of robotics, see, for example, Latombe~\cite{La} and Sharir \cite{Sh}.

A topological approach to this problem was developed by the second author \cite{Fa03}.  Let $PX$ denote the space of all continuous paths $\gamma\colon [0,1] \to X$, equipped with the compact-open topology.  The map $\pi\colon PX \to X \times X$ defined by sending a path to its endpoints, $\pi\colon\gamma \mapsto (\gamma(0),\gamma(1))$, is a fibration, with fiber $\Omega{X}$, the based loop space of $X$. The motion planning problem asks for a section of this fibration, a map $s\colon X\times X \to PX$ satisfying $\pi \circ s = \id_{X\times X}$.
It would be desirable for the motion planning algorithm to depend continuously on the input.  However, one can show that there exists a globally continuous section $s\colon X\times X \to PX$ if and only if $X$ is contractible, see \cite[Thm.~1]{Fa03}.

\begin{defn}
The \emph{topological complexity} of $X$, $\tc(X)$,  is the smallest positive integer $k$ for which $X\times X=U_1\cup\dots\cup U_k$, where $U_i$ is open and there exists a continuous section $s_i\colon U_i\to PX$ satisfying $\pi \circ s_i=\id_{U_i}$ for $1\le i \le k$.  In other words, the topological complexity of $X$ is
the Schwarz genus (or sectional category) of the path space fibration, $\tc(X)=\secat(\pi\colon PX \to X\times X)$.
\end{defn}

When faced with the problem of the motion of $n$ distinct particles in $X$ with the condition that no collision may occur during the motion, one is lead to consider the the topological complexity of the space
\[
F(X,n)=\set{(x_1,\dots,x_n) \in X^{\times n} \mid x_i \neq x_j\ \text{for}\ i \neq j},
\]
the configuration space of $n$ distinct ordered points in $X$, where $X^{\times n}=X\times \dots \times X$ is the Cartesian product of $n$ copies of $X$.
In the case where $X=\R^m$ is a Euclidean space, the topological complexity of the corresponding configuration space was found in \cite{FY,FG}.
The purpose of this note is to determine $\tc(F(X,n))$ 
in the case where $X=\Sigma_g$ is an orientable surface.

\begin{theo} The topological complexity of the configuration space of $n$ distinct ordered points on an orientable surface $\Sigma_g$ of genus $g$ is
\[
\tc(F(\Sigma_g,n))=\begin{cases}
3&\text{if $g=0$ and $n\le 2$,}\\
2n-2&\text{if $g=0$ and $n\ge 3$,}\\
2n+1&\text{if $g=1$ and $n\ge 1$,}\\
2n+3&\text{if $g\ge 2$ and $n\ge 1$.}
\end{cases}
\]
\end{theo}

As indicated, the topological complexity of the configuration space varies depending on the genus of the surface.  After recalling a number of necessary properties and discussing the technical tools we will use in \S\ref{sec:prelim}, we analyze the configuration space of the  sphere in \S\ref{sec:sphere}; the configuration space of the torus in \S\ref{sec:torus}; and the configuration space of a surface of higher genus in \S\ref{sec:big g}. 
We also determine the topological complexity of the configuration space of $n$ ordered points on a punctured surface, see \S\ref{sec:punctures}.  

For small values of $n$, our results summarize known facts concerning the topological complexity of orientable surfaces themselves. If $n=1$, we have $F(X,1)=X$ for any space $X$. Let $S^2$ be the $2$-dimensional sphere, $T=S^1\times S^1$ the torus, and denote a surface of genus $g \ge 2$ by $\Sigma=\Sigma_g$.  Then,
$\tc(S^2)=3$, $\tc(T)=3$, and $\tc(\Sigma)=5$
as indicated above.  Furthermore, there is a homotopy equivalence $F(S^2,2)\simeq S^2$, so
$\tc(F(S^2,2))=3$ as well.
Regarding punctured surfaces,
the complement of $m$ points on a surface has the homotopy type of a either a point~$*$
(with $\tc(*)=1$), a circle $S^1$ (with $\tc(S^1)=2$), or a bouquet $\bigvee_r S^1$ of $r\ge 2$ circles (with $\tc(\bigvee_r S^1)=3$).
Refer to the survey \cite{Fa05} for discussion of these and other relevant results.

It is interesting to compare the results stated above with the topological complexity of the Cartesian product $(\Sigma_g)^{\times n}$.  
Using the arguments described in \cite[pp. 115--116]{FInv}, one easily obtains
\[
\tc((\Sigma_g)^{\times n}) = \begin{cases}
2n+1 &\text{if $g=0$ or $g=1$,}\\
4n+1&\text{if $g\ge 2$}.
\end{cases}
\]
Thus, on a surface of high genus, the complexity of the collision-free motion planning problem for $n$ distinct points is roughly half of the complexity of the similar problem when the points are allowed to collide. This seems to contradict our intuitive understanding of the relative complexity of the two problems mentioned above, and may be viewed as a manifestation of the fact that the concept $\tc(X)$ reflects only part of the {\it \lq\lq real\rq\rq} complexity of the problem.

\section{Preliminaries} \label{sec:prelim}

In this section, we record a number of relevant properties of topological complexity, and discuss some relevant results on the cohomology of configuration spaces.

The topological complexity $\tc(X)$ depends only on the homotopy type of $X$, and satisfies the inequality
\begin{equation} \label{eq:upper bound}
\tc(X) \le 2 \dim X + 1,
\end{equation}
where $\dim(X)$ denotes the covering dimension of $X$, see \cite[\S\S2--3]{Fa03}. 

Recall that the spaces we consider in this paper are, up to homotopy type, finite CW complexes. For two such spaces $X$ and $Y$, as shown in \cite[Thm. 11]{Fa03}, the topological complexity of the Cartesian product admits the upper bound
\begin{equation} \label{eq:product}
\tc(X\times Y) \le \tc(X) + \tc(Y)-1.
\end{equation}

 If $A=\bigoplus_{k=0}^\ell A^k$ is a graded algebra over a field $\k$, with $A^k$ finite-dimensional for each $k$, define $\cl A$, the cup length of $A$, to be the largest integer $q$ for which there are homogeneous elements $a_1,\dots,a_q$ of positive degree in $A$ such that $a_1\cdots a_q \neq 0$.

The tensor product $A \otimes A$ has natural graded algebra structure, with multiplication given by $(u_1\otimes v_1)\cdot (u_2\otimes v_2)=(-1)^{|v_1|\cdot |u_2|} u_1 u_2 \otimes v_1 v_2$.  Multiplication in $A$ defines an algebra homomorphism $\mu\colon A \otimes A \to A$.  The zero-divisor cup length of $A$, denoted by $\zcl A$, is the cup length of the ideal $Z=\ker(\mu)$ of zero-divisors. It is a straightforward exercise to verify that the zero-divisor cup length has the properties listed below.

\begin{lemma} \label{lem:zcl estimates}
Let $A$ and $B$ be graded unital algebras over a field $\k$.
\begin{enumerate}
\item \label{item:subalgebra}
If $B$ is a subalgebra of $A$, then $\zcl A \ge \zcl B$.
\item \label{item:epi image}
If $B$ is an epimorphic image of $A$, then $\zcl A \ge \zcl B$.
\item \label{item:tensor product}
The tensor product $A \otimes B$ satisfies $\zcl A\otimes B \ge \zcl A + \zcl B$.
\end{enumerate}
\end{lemma}

By \cite[Thm.~7]{Fa03} the topological complexity of $X$ is larger than the {zero-divisor cup length} of the cohomology algebra $A=H^*(X;\k)$,
\begin{equation} \label{eq:lower bound}
\tc(X) \ge \zcl H^*(X;\k) + 1.
\end{equation}

For configuration spaces of ordered points on surfaces, topological considerations, together with the inequalities \eqref{eq:upper bound} and \eqref{eq:product}, yield upper bounds on the topological complexity.  These bounds are shown to be sharp using the cohomological lower bound \eqref{eq:lower bound}.  This requires some understanding of the structure of the cohomology ring of the configuration space.

The cohomology of the configuration space $F(X,n)$ of $n$ distinct ordered points in $X$ has been the object of a great deal of study, particularly for $X$ a manifold.  See, for instance, \cite{To} and the references therein. In the case where $X$ is a Euclidean space, the structure of the ring $H^*(F(\R^m,n);\Z)$ was determined by Arnol'd \cite{Ar} (for $n=2$) and Cohen \cite{Co}.  This ring has generators $\alpha_{i,j}^{}$, $1\le i,j\le n$, $i\neq j$, of degree $m-1$, and relations $\alpha_{i,j}^{}=(-1)^m\alpha_{j,i}^{}$, $\alpha_{i,j}^2=0$, and
$\alpha_{i,j}^{}\alpha_{i,k}^{}+\alpha_{j,k}^{}\alpha_{j,i}^{}+\alpha_{k,i}^{}\alpha_{k,j}^{}=0$
for distinct $i,j,k$. For general $X$, recall that $X^{\times n}=X\times \dots \times X$ ($n$ times) denotes the Cartesian product.  The structure of $H^*(F(\R^m,n);\Z)$
plays a significant role in the Leray spectral sequence of the inclusion $F(X,n) \hookrightarrow X^{\times n}$, analyzed by Cohen and Taylor \cite{CohenTaylor} and Totaro \cite{To}.  

Let $X$ be an oriented real manifold of dimension $m$, and let $\Delta \in H^m(X \times X;\k)$ be the cohomology class dual to the diagonal, where $\k$ is a field.  If $X$ is closed, $\omega \in H^m(X;\k)$ is a fixed generator, and $\set{\beta_i^{}}$ and $\set{\beta_i^*}$ are dual bases for $H^*(X;\k)$ satisfying $\beta_i^{}\cup \beta_j^* = \delta_{i,j}^{}\omega$, where $\delta_{i,j}^{}$ is the Kronecker symbol, the diagonal cohomology class may be expressed as
\[
\Delta=\sum (-1)^{|\beta_i^{}|}\beta_i^{} \times \beta_i^*,
\]
see \cite[Thm. 11.11]{Milnor}.

Let $p_i^{}\colon X^{\times n} \to X$ and $p_{i,j}^{}\colon X^{\times n} \to X \times X$ denote the natural projections, $p_i^{}(x_1,\dots,x_n)=x_i$ and $p_{i,j}^{}(x_1,\dots,x_n)=(x_i,x_j)$,
where $1\le i,j \le n$ and $i \neq j$.  Then, as shown by Cohen and Taylor \cite{CohenTaylor} and Totaro \cite{To}, the inclusion $F(X,n) \hookrightarrow X^{\times n}$ determines a Leray spectral sequence which converges to $H^*(F(X,n);\k)$. The initial term is the quotient of the algebra $H^*(X^{\times n};\k) \otimes H^*(F(\R^m,n);\k)$ by the relations $(p_i^*(x)-p_j^*(x))\otimes\alpha_{i,j}^{}$ for $i\neq j$ and $x \in H^*(X)$.  The first nontrivial differential is  given by $d \alpha_{i,j}^{}=p_{i,j}^*\Delta$.

The case where $X$ is a smooth, complex projective variety (specifically, a curve)
is of particular interest to us. In this instance, taking rational coefficients $\k=\Q$, Totaro \cite[Thm. 4]{To} shows that the differential in the spectral sequence described above is the only nontrivial differential.  Consequently,  the cohomology ring $H^*(F(X,n);\Q)$ is determined by $H^*(X;\Q)$.  The immediate degeneration of the spectral sequence also implies the following result, which we will use extensively.

\begin{proposition} \label{prop:totaro}
Let $X$ be a smooth projective variety
over $\C$ of real dimension $m$, let $H=H^*(X^{\times n};\Q)$, and let $I$ be the ideal in $H$ generated by the elements $$\Delta_{i,j}^{}=p^*_{i,j} \Delta\in H^{m}(X^{\times n};\Q)$$ for all $1\le i<j \le n$.  Then the quotient $H/I$ is a subalgebra of the rational cohomology ring of the configuration space $F(X,n)$. Thus, using Lemma \ref{lem:zcl estimates}, one obtains
\begin{equation*}
\tc(F(X, n)) \ge \zcl H/I +1.
\end{equation*}
\end{proposition}

We will also make frequent use of the classical Fadell--Neuwirth theorem \cite[Thm.~3]{FN}, which shows that for $m < n$, the projection $F(X,n) \to F(X,m)$ onto the first $m$ coordinates is a fibration, with fiber $F(X\smallsetminus Q_m,n-m)$.  Here, and throughout the paper, we use the symbol $Q_m$ to denote a collection of $m\ge 1$ distinct points in the topological space $X$.

\section{Genus zero} \label{sec:sphere}
In this section, we determine the topological complexity of the configuration space of $n$ ordered points on the sphere $S^2$.  Since $F(S^2,1)=S^2$ and $F(S^2,2)$ is equivalent to the tangent bundle over $S^2$, hence has the homotopy type of $S^2$, we have $\tc(F(S^2,n))=3$ for $n\le 2$.  For larger $n$, the following holds.

\begin{theorem} \label{thm:sphere}
For $n\ge 3$, the topological complexity of the configuration space of $n$ distinct ordered points on the sphere is
\[
\tc(F(S^2,n))=2n-2.
\]
\end{theorem}
\begin{proof}
Fix a triple of pairwise distinct points $P_1, P_2, P_3\in S^2$, and consider the group $\PSL(2,\C)$ of M\"obius transformations acting on the sphere $S^2=\PP^1$. For any ordered triple of pairwise distinct points $(x_1, x_2, x_3)\in S^2\times S^2\times S^2$, there is a unique M\"obius transformation $A\colon S^2 \to S^2$ satisfying $Ax_i=P_i$ for $i=1, 2, 3$, and depending continuously on $x_1,x_2,x_3$.  
This yields homeomorphisms $F(S^2,3) \to \PSL(2,\C)$, $(x_1,x_2,x_3)\mapsto A$, and, 
for $n\ge 4$, 
\[
F(S^2, n)\to \PSL(2, \C) \times F(S^2\smallsetminus Q_3, n-3),
\]
sending a configuration $(x_1, \dots, x_n)$ to $(A, (y_4, \dots, y_{n}))$. Here $Q_3=\{P_1, P_2, P_3\}$, the transformation $A$ is defined as above, and $y_i=Ax_i$ for $i=4, \dots, n$.
Noting that $\PSL(2,\C)$ deformation retracts onto $\SO(3)$ and  that $F(S^2\smallsetminus Q_3) =F(\R^2-Q_2, n-3)$, we obtain homotopy equivalences 
$F(S^2,3)\simeq \SO(3)$, and, for $n\ge 4$,  
\begin{equation} \label{eq:sphere htpy}
F(S^2,n) \simeq \SO(3) \times F(\R^2\smallsetminus Q_2,n-3).
\end{equation}

The topological complexity of the (connected) Lie group $\SO(3)$ is
$\tc(SO(3))=\cat(\SO(3))=4$, see \cite[Lem. 8.2]{Fa04}.  
This finishes the proof for $n=3$.  For $n\ge 4$, it is also 
known that $\tc(F(\R^2\smallsetminus Q_2,n-3))=2n-5$, see \cite[Thm. 6.1]{FGY}.
These facts and the product inequality \eqref{eq:product} imply
\[
\tc(F(S^2,n)) \le \tc(\SO(3)) + \tc(F(\R^2\smallsetminus Q_2,n-3)-1 = 4 + 2n-5 -1=2n-2.
\]

To complete the proof, by \eqref{eq:lower bound} it suffices to show that $\zcl H^*(F(S^2,n);\Z_2) \ge 2n-3$.  As indicated, we will consider cohomology with $\Z_2$ coefficients. From the homotopy equivalence \eqref{eq:sphere htpy}, we have
\[
H^*(F(S^2,n);\Z_2) \cong H^*(\SO(3);\Z_2) \otimes H^*(F(\R^2\smallsetminus Q_2,n-3);\Z_2).
\]
As noted in Lemma \ref{lem:zcl estimates}, it follows that
\[
\zcl H^*(F(S^2,n);\Z_2) \ge \zcl H^*(\SO(3);\Z_2) \otimes \zcl H^*(F(\R^2\smallsetminus Q_2,n-3);\Z_2).
\]
It is readily checked that the zero-divisor cup length of $H^*(\SO(3);\Z_2) \cong \Z_2[\alpha]/(\alpha^4)$ is equal to $3$.  For the configuration space, the zero-divisor cup length of the integral cohomology ring $H^*(F(\R^2\smallsetminus Q_2,n-3);\Z)$ was computed in \cite[\S6]{FGY}.  Repeating this computation with $\Z_2$ coefficients yields the same result,
\[
\zcl H^*(F(\R^2\smallsetminus Q_2,n-3);\Z_2) =\zcl H^*(F(\R^2\smallsetminus Q_2,n-3);\Z)=2(n-3).
\]
It follows that $\zcl H^*(F(S^2,n)) \ge 3 + 2(n-3)=2n-3$, as required.
\end{proof}

\begin{remark} For $n\ge 3$, the topological complexity of $F(S^2,n)$ coincides with that of $F(\R^2,n)$, the configuration space of $n$ distinct ordered points in the plane.  See Farber and Yuzvinsky \cite{FY} for the calculation of the latter.
\end{remark}

\section{Genus one} \label{sec:torus}

\begin{theorem} \label{thm:torus}
The topological complexity of the configuration space of $n$ distinct ordered points on the torus is
\[
\tc(F(T,n))=2n+1.
\]
\end{theorem}
\begin{proof}
For $n=1$, since $F(T,1)=T$, we have $\tc(F(T,1))=3$ as noted previously.   So assume that $n\ge 2$.

Since $T=S^1 \times S^1$ is a group, we have
\[
F(T,n) \cong T \times F(T\smallsetminus Q_1,n-1).
\]
Explicitly, view $S^1$ as the set of complex numbers of length one, and let $Q_1=
\set{(1,1)} \in T$.  It is then readily checked that the map $T \times F(T\smallsetminus Q_1,n-1) \to F(T,n)$ defined by
\[
 \bigl((u,v),\bigl((z_1,w_1),\dots,(z_{n-1},w_{n-1})\bigr)\bigr) \mapsto
\bigl((u,v),(uz_1,vw_1),\dots,(uz_{n-1},vw_{n-1})\bigr)
\]
is a homeomorphism.

Using Fadell--Neuwirth fibrations, one can show that
$F(T\smallsetminus Q_1,n-1)$ is an Eilenberg-Mac\,Lane space of type $K(G,1)$,
where $G=\pi_1(F(T\smallsetminus Q_1,n-1))$ is the pure braid group of $T\smallsetminus Q_1$.  Since the group $G$ is an iterated semidirect product of free groups \cite{Be,GG}, the space $F(T\smallsetminus Q_1,n-1)$ has the homotopy type of a cell complex of dimension $n-1$, see \cite[\S1.3]{CSchain}. So $\tc(F(T\smallsetminus Q_1,n-1))\le 2n-1$ by \eqref{eq:upper bound}, and the product inequality \eqref{eq:product} yields
\[
\tc(F(T,n)) \le \tc(T) + \tc(F(T\smallsetminus Q_1,n-1)) = 3 + 2n-1 - 1 = 2n+1.
\]

By \eqref{eq:lower bound}, it suffices to show that $\zcl H^*(F(T,n);\Q) \ge 2n$.  We establish this using
the Leray spectral sequence of the inclusion $F(T,n) \hookrightarrow T^{\times n}$ developed by Totaro~\cite{To}.
Since we use rational coefficients throughout the argument, we subsequently suppress coefficients and write $H^*(X)=H^*(X;\Q)$ for brevity. Let $a,b\in H^1(T)$ be the generators of $H^*(T)$.  Note that the diagonal class $\Delta \in H^2(T\times T)$ is given by
\[ 
\Delta=ab\times 1 + 1\times ab + b\times a-a\times b = (1\times a-a\times 1)(1\times b-b\times 1).
\]

Let $H_T=H^*(T^{\times n})=[H^*(T)]^{\otimes n}$. Note that $H_T$ is an exterior algebra.
Denote the generators of $H_T$ by $a_i,b_i$, $1\le i \le n$, where $u_i=1\times\dots\times \overset{i}{u} \times \dots \times 1$. Let $I_T$ be the ideal in $H_T$ generated by the elements $\Delta_{i,j}=p^*_{i,j}\Delta$, $1\le i<j\le n$, and observe that, in this notation, we have
\[
\Delta_{i,j}=(a_j-a_i)(b_j-b_i).
\]

Realizing $T$ as a smooth, complex projective curve, Proposition \ref{prop:totaro} implies that the algebra
\begin{equation} \label{eq:AT}
A_T=H_T/I_T
\end{equation}
is a subalgebra of $H^*(F(T,n))$.  Since $\zcl H^*(F(T,n))\ge \zcl A_T$ by Lemma \ref{lem:zcl estimates},
it is enough to show that $\zcl A_T\ge 2n$.

Introduce a new basis for $H_T$ as follows:
\[
x_j=\begin{cases}
a_1&\text{if $j=1$,}\\
a_j-a_1&\text{if $2\le j\le n$,}
\end{cases}
\qquad
y_j=\begin{cases}
b_1&\text{if $j=1$,}\\
b_j-b_1&\text{if $2\le j\le n$.}
\end{cases}
\]
In this basis, $\Delta_{1,j}=x_j y_j$ and $\Delta_{i,j}=x_jy_j-x_jy_i-x_iy_j+x_iy_i$ for $i>1$.
Consequently, the ideal $I_T$ is given by
\begin{equation}\label{eq:xygens}
I_T=\langle x_jy_j\ (2\le j \le n),\ x_jy_i+x_iy_j\ (2\le i<j \le n)\rangle
\end{equation}
Since $I_T$ is generated in degree $2$, we abuse notation and 
denote the generators of $A_T=H_T/I_T$ by $x_i,y_i$, $1\le i\le n$.
From the description \eqref{eq:xygens} of the ideal $I_T$, monomials in $A_T$ cannot have any repetition of indices (at least $2$).  Additionally, the presence of $x_jy_i+x_iy_j$ in $I_T$ implies that any monomial in $A_T$ may be expressed, up to sign, as a monomial in which all $x$-indices are smaller than all $y$-indices (with the exception of~$1$).  Since such expressions are unique and nonzero in $A_T$, this algebra has basis
\begin{equation} \label{eq:A basis}
\set{x_1^{\epsilon_x}y_1^{\epsilon_y}x_J^{}y_K^{} \mid
\epsilon_x,\epsilon_y \in\set{0,1},\ J,K\subset [2,n],\ \max J<\min K},
\end{equation}
where $[2,n]=\set{2,3,\dots,n}$ and, for instance, $x_J=x_{j_1}\cdots x_{j_p}$ if $J=\set{j_1,\dots,j_p}$.

We now complete the proof by showing that the zero-divisor cup length of $A_T$ is $2n$.  Consider the zero-divisors $\bar{x}_j=x_j\otimes1-1\otimes x_j$ and $\bar{y}_j=y_j\otimes1-1\otimes y_j$, $1\le j\le n$, in $A_T\otimes A_T$.  We claim that their product is nonzero.  Note that $\bar{x}_j\bar{y}_j=y_j\otimes x_j-x_j\otimes y_j$ if $2\le j \le n$, while $\bar{x}_1\bar{y}_1=x_1y_1\otimes 1+
y_1\otimes x_1-x_1\otimes y_1+1\otimes x_1y_1$. So we have
\[
\prod_{j=1}^n \bar{x}_j\bar{y}_j=\bar{x}_1\bar{y}_1\prod_{j=2}^n(y_j\otimes x_j-x_j\otimes y_j)
=\bar{x}_1\bar{y}_1\sum_{J\subset [2,n]} \epsilon_J y_J x_{J^c}\otimes y_{J^c} x_J,
\]
where $\epsilon_J \in \set{1,-1}$ and $J^c=[2,n]\smallsetminus J$. In particular, the above sum includes the summand $(-1)^{n}y_2y_3\cdots y_n\otimes x_2x_3\cdots x_n$ which cannot arise when other summands are expressed in terms of the specified basis \eqref{eq:A basis} for $A_T$.
Consequently, expanding the product $\prod_{j=1}^n \bar{x}_j\bar{y}_j$ yields a summand 
$\pm  y_1 y_2y_3\cdots y_n\otimes x_1x_2x_3\cdots x_n$, 
and no other summand in the expansion involves this tensor product of basis elements.  Thus,
$\prod_{j=1}^n \bar{x}_j\bar{y}_j$ is nonzero in $A_T\otimes A_T$, as asserted.
\end{proof}

\begin{remark} The subalgebra $A_T=H_T/I_T$ is not isomorphic to $H^*(F(T,n))$.  One can check, for instance, that the differential $d\colon E_2^{1,1}\to E_2^{3,0}$ has nontrivial kernel, 
where $E_2^{1,1}$ is the quotient of $H^1(F(T,n))\otimes H^1(F(\R^2,n))$ by the relations $(p_i^*(x)-p_j^*(x))\otimes \alpha_{i,j}$ for $x\in H^1(T)$ and $E_2^{3,0}=H^3(F(T,n))$. 
However, $A_T$ and $H^*(F(T,n))$ do have the same zero-divisor cup length.  Theorem \ref{thm:torus} implies that $\zcl A_T=\zcl H^*(F(T,n))=2n$.
\end{remark}

\begin{remark} The algebra $A_T=H_T/I_T$ is Koszul.  A straightforward application of the Buchberger criterion (see \cite[Thm. 1.4]{AHH}) reveals that the generating set \eqref{eq:xygens} of the defining ideal $I_T$ is a Gr\"obner basis.  Since all of these generators are of degree two, the Koszulity of $A_T$ follows (see, for instance, \cite[Thm. 6.16]{Yuz}).
\end{remark}

\section{Higher genus} \label{sec:big g}

\begin{theorem} \label{thm:big g}
The topological complexity of the configuration space of $n$ distinct ordered points on a surface
$\Sigma$ of genus $g \ge 2$ is
\[
\tc(F(\Sigma,n))=2n+3.
\]
\end{theorem}
\begin{proof}
For $n=1$, since $F(\Sigma,1)=\Sigma$, we have $\tc(F(\Sigma,1))=5$ as noted previously.  So assume that $n\ge 2$.

The configuration space $F(\Sigma,n)$ is an Eilenberg-Mac\,Lane space of type
$K(G,1)$, where $G=\pi_1(F(\Sigma,n))$ is the pure braid group of $\Sigma$.  Since the Fadell--Neuwirth fibration $F(\Sigma,n)\to\Sigma$ has a section, the group $G\cong \pi_1(F(\Sigma\smallsetminus Q_1,n-1)) \rtimes \pi_1(\Sigma)$ is a semidirect product. As in the genus $1$ case, the group $\pi_1(F(\Sigma\smallsetminus Q_1,n-1))$ is an iterated semidirect product of free groups.  It follows that the cohomological dimension of $G$ is equal to $n+1$, as is the geometric dimension.  Consequently, $F(\Sigma,n)$ has the homotopy type of a cell complex of dimension $n+1$.  So $\tc(F(\Sigma,n)) \le 2n+3$.

By 
\eqref{eq:lower bound}, it suffices to show that $\zcl H^*(F(\Sigma,n);\Q) \ge 2n+2$.  We again use the
Leray spectral sequence of the inclusion $F(\Sigma,n) \hookrightarrow \Sigma^{\times n}$ following \cite{To},
and write $H^*(\Sigma)=H^*(\Sigma;\Q)$.  Let $a(p)$, $b(p)$, $1\le p \le g$, be the generators of $H^1(\Sigma)$, satisfying, for $p\neq q$,  $a(p)b(p)=a(q)b(q)=\omega$ and  $a(p)a(q)=b(p)b(q)=a(p)b(q)=0$, where $\omega$ generates $H^2(\Sigma)$.  Then the diagonal class $\Delta \in H^2(\Sigma \times \Sigma)$ may be expressed as
\[
\Delta=\omega\times 1 + 1\times \omega +
\sum_{p=1}^g(b(p)\times a(p)-a(p)\times b(p)).
\]

Let $H_\Sigma=H^*(\Sigma^{\times n})=[H^*(\Sigma)]^{\otimes n}$, and let $I_\Sigma$ be the ideal in $H_\Sigma$ generated by the elements $\Delta_{i,j}=p^*_{i,j}\Delta$, $1\le i<j\le n$.  Realizing $\Sigma$ as a smooth, complex projective curve, by Proposition \ref{prop:totaro} and Lemma \ref{lem:zcl estimates},
it is enough to show that the subalgebra $A_\Sigma=H_\Sigma/I_\Sigma$ of $H^*(F(\Sigma,n))$ satisfies $\zcl A_\Sigma \ge 2n+2$.  Annihilating all generators of $H_\Sigma$ of the form $1\times\dots\times u \times\dots\times 1$, where $u \in\set{a(q),b(q)\mid 3\le q\le n}$, it suffices to consider the genus $g=2$ case.

For a genus two surface $\Sigma$, 
denote the generators of $H^1(\Sigma)$ by $a,b,c,d$, with $ab=cd=\omega$, and other cup products equal to zero.  Let $a_i,b_i,c_i,d_i$ be the generators of $H_\Sigma$, where for $1\le i\le n$, $u_i=1\times\dots\times \overset{i}{u} \times \dots \times 1$ as before.  In this notation, we have
\[
\Delta_{i,j}=\omega_i  + \omega_j - a_i b_j + b_i a_j - c_i d_j + d_i c_j.
\]

Consider the ideal $J_\Sigma=\langle c_ic_j,c_id_j,d_id_j\mid 1\le i,j\le n, i\neq j\rangle$ in $H_\Sigma$.  Observe that
\[
\Delta_{i,j}=a_ib_i  +  a_jb_j - a_i b_j + b_i a_j=(a_j-a_i)(b_j-b_i)\mod J,
\]
and that 
\[
B_\Sigma = H_\Sigma/(I_\Sigma+J_\Sigma)\cong (H_\Sigma/I_\Sigma)/((I_\Sigma+J_\Sigma)/I_\Sigma)\]
is a quotient of $A_\Sigma$.  Consequently,  $\zcl B_\Sigma \le \zcl A_\Sigma$.  The subalgebra of $B_\Sigma$ generated by $\set{a_i,b_i\mid 1\le i \le n}$ is isomorphic to the algebra $A_T$ arising in the genus one case, see \eqref{eq:AT}.  Letting $x_j=a_j-a_1$ and $y_j=b_j-b_1$ for $j\ge 2$ as in that case, it follows that the set
\[
 \set{x_Jy_K \mid J,K\subset [2,n], \max J < \min K}
\]
is linearly independent in $B_\Sigma$.  From this, it follows easily that the zero-divisor cup length of $B_\Sigma$ is (at least) $2n+2$.  Indeed, writing $\bar{u}=u\otimes 1-1\otimes u \in B_\Sigma \otimes B_\Sigma$ for $u\in B_\Sigma$, a calculation reveals that
$\bar{a}_1\bar{b}_1\bar{c}_1\bar{d}_1=2\omega_1 \otimes \omega_1$.
Then, expanding the product $\bar{a}_1\bar{b}_1\bar{c}_1\bar{d}_1 \prod_{j=2}^n \bar{x}_j \bar{y}_j$ of $2n+2$ zero divisors yields a summand
\[
 \pm 2\omega_1 y_2y_3\cdots y_n\otimes \omega_1 x_2x_3\cdots x_n,
\]
and no other summand in the expansion involves this (nonzero) tensor product.  Thus, $2n+2 \le \zcl B_\Sigma \le \zcl A_\Sigma \le \zcl H^*(F(\Sigma,n))$.
\end{proof}

\section{Punctured surfaces} \label{sec:punctures}
In this section, we determine the topological complexity of the configuration space of $n$ ordered points on a punctured surface. Recall that
$X\smallsetminus Q_m$ denotes the complement of a set $Q_m$ of $m$ distinct  points in $X$.

\begin{theorem} \label{thm:punctured sphere}
For $m \ge 1$, the topological complexity of the configuration space of $n$ distinct ordered points on $S^2\smallsetminus  Q_m$ is
\[
\tc(F(S^2\smallsetminus Q_m,n))=\begin{cases}
1&\text{if $m=1$ and $n=1$,}\\
2n-2&\text{if $m=1$ and $n\ge 2$,}\\
2n&\text{if $m=2$ and $n\ge 1$,}\\
2n+1&\text{if $m\ge 3$ and $n\ge 1$.}
\end{cases}
\]
\end{theorem}
\begin{proof}
Observe that $F(S^2\smallsetminus Q_1,n)=F(\R^2,n)$, that $F(S^2\smallsetminus Q_2,n)=F(\R^2\smallsetminus Q_1,n) \simeq F(\R^2,n+1)$, and that for $m\ge 3$,  $F(S^2\smallsetminus Q_m,n)=F(R^2\smallsetminus Q_{m-1},n)$ is the configuration space of $n$ points in the complement of at least $2$ points in $\R^2$.  For $n=1$, $F(S^2\smallsetminus Q_m,1)$ has the homotopy type of a bouquet of $m-1$ circles (where a bouquet of $0$ circles is a point), and the theorem follows easily.  For $n\ge 2$, in light of the above observations,
the theorem follows from results of Farber and Yuzvinsky \cite{FY} for $m\le 2$, and from results of Farber, Grant, and Yuzvinsky \cite{FGY} for $m\ge 3$.
\end{proof}

\begin{remark}
For $m\ge 1$, the configuration space $F(S^2\smallsetminus Q_m,n))$ is an Eilenberg-Mac\,Lane space of type $K(\pi,1)$, where $\pi=\pi_1(F(S^2\smallsetminus Q_m,n))$ is the pure braid group of $S^2\smallsetminus  Q_m$.  Since these groups are almost-direct product of free groups, the above result may also be obtained using the methods of \cite{DC}.
\end{remark}

\begin{theorem} \label{conj:big g punctured}
Let $\Sigma$ be a surface of genus $g \ge 1$. For $m \ge 1$, the topological complexity of the configuration space of $n$ distinct ordered points on $\Sigma\smallsetminus  Q_m$ is
\[
\tc(F(\Sigma\smallsetminus  Q_m,n))=2n+1.
\]
\end{theorem}
\begin{proof}
For $n=1$, since $F(\Sigma\smallsetminus Q_m,1)=\Sigma\smallsetminus Q_m$ has the homotopy type of a bouquet of $r\ge 2$ circles, we have $\tc(F(\Sigma\smallsetminus Q_m,1))=3$.  So assume that $n\ge 2$.

The configuration space $F(\Sigma\smallsetminus Q_m,n)$ is an Eilenberg-Mac\,Lane space of type
$K(G,1)$, where $G=\pi_1(F(\Sigma\smallsetminus Q_m,n))$ is the pure braid group of $\Sigma\smallsetminus Q_m$.  As in previous cases, the group $G$ is an iterated semidirect product of free groups, and the geometric dimension of $G$ is equal to $n$.  Consequently, $F(\Sigma\smallsetminus Q_m,n)$ has the homotopy type of a cell complex of dimension $n$.  So $\tc(F(\Sigma\smallsetminus Q_m,n)) \le 2n+1$.

By \eqref{eq:lower bound}, it suffices to show that $\zcl H^*(F(\Sigma\smallsetminus Q_m,n);\k) \ge 2n$.  We will use complex coefficients $\k=\C$, and write $H^*(X)=H^*(X;\C)$.

First, assume that $m=1$.  Let $p\in \Sigma$ and $Q_1=\set{p}$.  The configuration space $F(\Sigma\smallsetminus Q_1,n)$ may be realized as
$F(\Sigma\smallsetminus Q_1,n)=X\smallsetminus D=X\smallsetminus \bigcup_{i=1}^n D_i$, where $X=F(\Sigma,n)$ and $D_i= \set{(x_1,\dots,x_n) \in X \mid x_i=p}$. 
Note that $D_i \cong F(\Sigma\smallsetminus Q_1,n-1)$ is closed in $X$, and $D_i \cap D_j = \emptyset$ if $i\neq j$.
Consider the corresponding Gysin sequence
\begin{equation} \label{eq:gysin}
\ldots \xrightarrow{\ } H^{k-2}(D) \xrightarrow{\,\delta\,} H^k(X) \xrightarrow{\,j^*} H^k(X\smallsetminus D) \xrightarrow{\,R\,} H^{k-1}(D) \xrightarrow{\,\delta\,} H^{k+1}(X) \xrightarrow{\ } \ldots
\end{equation}
where $j^*$ is induced by the inclusion $j\colon X\smallsetminus D \hookrightarrow X$, $R$ is the residue map, and $\delta$ is the connecting homomorphism.  Using this sequence, we observe that the map $j^*\colon H^1(X) \to H^1(X\smallsetminus D)$, that is,
$j^*\colon H^1(F(\Sigma,n)) \to H^1(F(\Sigma\smallsetminus Q_1,n))$, is injective.

Let $H_\Sigma=H^*(\Sigma^{\times n})=[H^*(\Sigma)]^{\otimes n}$, and let $I_\Sigma$ be the ideal in $H_\Sigma$ generated by the elements $\Delta_{i,j}=p^*_{i,j}\Delta$, $1\le i<j\le n$.  Then, 
by Proposition \ref{prop:totaro}, $A_\Sigma=H_\Sigma/I_\Sigma$ is a subalgebra of $H^*(F(\Sigma,n))=H^*(X)$.  Since $I_\Sigma$ is generated in degree $2$, the generators of $H_\Sigma$ are among the generators of $H^*(X)$.  Consider the generators $x_1=a_1$, $y_1=b_1$, $x_i=a_i-a_1$, $y_i=b_i-b_1$, $2\le i \le n$, which arose the the proofs of Theorems \ref{thm:torus} and \ref{thm:big g}.  The above observation implies that their images under the map $j^*$ are among the generators of $H^1(X\smallsetminus D)=H^1(F(\Sigma\smallsetminus Q_1,n))$.

Write $u_i=j^*(x_i)$ and $v_i=j^*(y_i)$, and let $\bar{u}_i=u_i\otimes 1-1\otimes u_i$ and 
$\bar{v}_i=v_i\otimes 1-1\otimes v_i$ be the corresponding zero-divisors in $H^*(X\smallsetminus D)
\otimes H^*(X\smallsetminus D)$.
We will show that the product of these $2n$ zero-divisors is nonzero using mixed Hodge structures.  
References include \cite{Dimca}, \cite{GS}, and \cite{PS}.

Realizing $\Sigma$ as a smooth projective variety, the Hodge structure on $H^*(\Sigma)$ is pure.  Assume without loss that the elements $a$ and $b$ of $H^1(\Sigma)$ are of types $(1,0)$ and $(0,1)$ respectively.  Then, for each $i$, $x_i$ and $y_i$ are of types $(1,0)$ and $(0,1)$.  Since the map $j^*\colon H^*(X) \to H^*(X\smallsetminus D)$ preserves type,
the elements $u_i$ and $v_i$ of $H^1(X\smallsetminus D)$ are of types $(1,0)$ and $(0,1)$.  Furthermore, each of these elements is of weight $1$.

Since $X=F(\Sigma,n)$ is smooth, for each $m$, the weight filtration on $H^m(X)$ satisfies $0=W_{m-1}(H^m(X)) \subset W_m(H^m(X))=f^*(H^m(\bar{X}))$, where $f\colon X \hookrightarrow \bar{X}$ is any compactification, see \cite[Thm. C24]{Dimca}. Taking $\bar{X}= \Sigma^{\times n}$, we have $W_m(H^m(X))=f^*(H^m(\Sigma^{\times n}))$.  It follows that $A_\Sigma^m \subset W_m(H^m(X))$. 
Recall that the divisor $D$ has disjoint components $D_i=\set{(x_1,\dots,x_n)\in X\mid x_i=p}$, let $\Sigma^{\times n-1}_i=\set{(x_1,\dots,x_n)\in \Sigma^{\times n}\mid x_i=p}$, and let
$f_i\colon D_i \hookrightarrow \Sigma^{\times n-1}_i$.  Then, $W_{m}(H^{m}(D_i))=f_i^*(H^{m}(\Sigma^{\times n-1}))$.  Note that the diagram
\begin{equation*}
\xymatrixcolsep{25pt}
\xymatrix{
D_i \ar[d]_{f_i} \ar[r] & X \ar[d]^{f}
\\
\qquad \Sigma^{\times n-1}_i \ar[r] & \quad \Sigma^{\times n}
}
\end{equation*}
commutes, where the horizontal maps are inclusions.

The Gysin mapping $H^{k-2}(\Sigma^{\times n-1}_i) \to H^k(\Sigma^{\times n})$ is obtained by applying Poincar\'e duality to the map $H_{2n-k}(\Sigma^{\times n-1}_i) \to H_{2n-k}(\Sigma^{\times n})$ induced by inclusion, see \cite[\S5]{GS}.  
The ring $H^*(\Sigma^{\times n-1}_i)$ is generated by the (images of the) generators $u_j$ of 
$H^*(\Sigma^{\times n}))$ which do not involve the index $i$.  In terms of these generators, 
one can check that this
map is, up to sign, multiplication by $\omega_i=a_i b_i$.  
As noted in \cite[Rem.~C30]{Dimca}, the connecting homomorphism $\delta\colon H^{n-2}(D) \to H^{n}(X)$ in the Gysin sequence \eqref{eq:gysin} is a morphism of mixed Hodge structures of type $(1,1)$.  It follows, by functoriality, that the restriction $\delta \colon W_{m-2}(H^{m-2}(D_i))
\to W_m(H^m(X))$ is also multiplication by $a_i b_i$.

These considerations imply that the image of $\delta\colon W_{n-2}(H^{n-2}(D)) \to W_n(H^n(X))$ is contained in the ideal $\langle a_i b_i \mid 1 \le i \le n\rangle$.  In terms of the generators $x_i$ and $y_i$, this ideal is generated by $x_1y_1$ and $x_iy_1+x_1y_i$ for $2\le i \le n$.  Observe that the basis elements $x_1\cdots x_k y_{k+1}\cdots y_n$ of $A_\Sigma$ from \eqref{eq:A basis} are nontrivial modulo this ideal.  Consequently, the corresponding elements $u_1\cdots u_k v_{k+1}\cdots v_n=j^*(x_1\cdots x_k y_{k+1}\cdots y_n)$ of $H^n(X\smallsetminus D)$ are nonzero, and are linearly independent since they are of distinct types $(k,n-k)$.

It follows easily that the product $\prod_{i=1}^n \bar{u_i} \bar{v_i}$ is nonzero in 
$H^*(X\smallsetminus D)\otimes H^*(X\smallsetminus D)$. Expanding this product yields a linear combination of terms of the form $\alpha \otimes \beta$, where $\alpha$ and $\beta$ are among the independent elements $u_1\cdots u_k v_{k+1}\cdots v_n$ noted above.  In particular, there is a single summand
$\pm u_1u_2\cdots u_n \otimes v_1v_2\cdots v_n$,
and no other summand in the expansion involves this (nonzero) tensor product. Thus, $\zcl H^*(X\smallsetminus D) = \zcl H^*(F(\Sigma\smallsetminus Q_1,n)) \ge 2n$ in the case $m=1$.

For $m>1$, assume inductively that the image of the product $\prod_{i=1}^n \bar{x}_i \bar{y}_i$ is nonzero in $H^*(F(\Sigma\smallsetminus Q_{m-1},n))\otimes H^*(F(\Sigma\smallsetminus Q_{m-1},n))$.  Write $F(\Sigma\smallsetminus Q_{m},n) = F(\Sigma\smallsetminus Q_{m-1},n)\smallsetminus D$ in a manner analogous to the case $m=1$. Then, arguing as above reveals that the image of this product is nonzero in
$H^*(F(\Sigma\smallsetminus Q_{m},n))\otimes H^*(F(\Sigma\smallsetminus Q_{m},n))$ as well.
Thus, $\zcl H^*(F(\Sigma\smallsetminus Q_m,n)) \ge 2n$ as required.
\end{proof}

\section*{Acknowledgment}
Portions of this work were carried out during the Fall of 2008, when the first author visited the
Department of Mathematical Sciences at the University of Durham.  He thanks the Department for its hospitality, and for providing a productive mathematical environment.

\newcommand{\arxiv}[1]{{\texttt{\href{http://arxiv.org/abs/#1}{{arXiv:#1}}}}}

\newcommand{\MRh}[1]{\href{http://www.ams.org/mathscinet-getitem?mr=#1}{MR#1}}

\end{document}